\documentclass[12pt]{amsart}
\usepackage{amsfonts,amssymb,amsmath,amsthm,bm,mathrsfs,amscd}
\usepackage{color,graphicx}
\usepackage{url}
\usepackage{enumerate}
\usepackage{extsizes}

\newtheorem{Them}{Theorem}[section]
\newtheorem{Lem}[Them]{Lemma}
\newtheorem{Prop}[Them]{Proposition}

\theoremstyle{definition}
\newtheorem{Defi}[Them]{Definition}

\numberwithin{equation}{section}

\newcommand{\Hom}{\operatorname{Hom}}

\newcommand{\CB}{\mathbb{C}}

\newcommand{\FC}{\mathcal{F}}
\newcommand{\GC}{\mathcal{G}}

\newcommand{\ZB}{\mathbb{Z}}

\newcommand{\HC}{\mathcal{H}}

\author{Norihiro Nakashima}
\address{
School of Information Environment\\
Tokyo Denki University\\
Inzai, 270-1382, Japan}
\email{nakashima@mail.dendai.ac.jp}
\thanks{}

\author{Hiroaki Terao}
\address{
Department of Mathematics\\
Hokkaido University\\
Sapporo, 060-0810, Japan}
\email{hterao00@za3.so-net.ne.jp}
\thanks{}

\author{Shuhei Tsujie}
\address{
Department of Mathematics\\
Hokkaido University\\
Sapporo, 060-0810, Japan}
\email{tsujie@math.sci.hokudai.ac.jp}
\thanks{}

\keywords{basic invariants; invariant theory;
finite unitary reflection groups}
\subjclass{Primary 13A50; Secondary 20F55}
\begin{document}

\title[Canonical systems of basic invariants]{Canonical systems of basic invariants for unitary reflection groups}

\begin{abstract}
It has been known that there exists a canonical system for every
finite real reflection group. The first and the third authors obtained
an explicit formula for a canonical system in the previous paper.
In this article, we first define canonical systems for the finite
{\it unitary} reflection groups, and then prove their existence.
Our proof does not depend on the classification of unitary reflection groups.
Furthermore, we give an explicit formula for a canonical system
for every unitary reflection group.
\end{abstract}
\maketitle

\section{Introduction}
Let $V$ be an $n$-dimensional unitary space, and
$W\subseteq U(V)$ a finite unitary reflection group.
Each reflection fixes a hyperplane in $V$ pointwise.
Let $S$ denote the symmetric algebra $S(V^{\ast})$ of the
dual space $V^{\ast}$, and $S_{k}$ the vector space
consisting of homogeneous polynomials of degree $k$
together with the zero polynomial.
Then $W$ acts contravariantly on $S$ by
$(w\cdot f)(v)=f(w^{-1}\cdot v)$ for $v\in V$, $w\in W$ and $f\in S$.
The action of $W$ on $S$ preserves the degree of
homogeneous polynomials, and $W$ acts also on $S_{k}$.
The subalgebra $R=S^{W}$
of $W$-invariant polynomials of $S$ is generated by
$n$ algebraically independent homogeneous polynomials by Chevalley \cite{Che}.
A system of such generators is called a system of {\bf
basic invariants} of $R$.

Let $x_{1},\dots,x_{n}$ be an orthonormal basis for $V^{\ast}$,
and $\partial_{1},\dots,\partial_{n}$ the basis for
$V^{\ast\ast}$ dual to $x_{1},\dots,x_{n}$.
The symmetric algebra of $V^{\ast\ast}$ acts naturally on $S$
as differential operators (e.g., Kane \cite[$\S$25-2]{Kane}).
Let $\overline{c}$ denote the complex conjugate of $c\in\CB$.
For $f=\sum_{\bm{a}}c_{\bm{a}}x^{\bm{a}}\in S$,
a differential operator $f^{\ast}$ is defined by
\begin{align}
f^{\ast}:=\overline{f}(\partial):=
\sum_{\bm{a}}\overline{c_{\bm{a}}}\partial^{\bm{a}},
\end{align}
where $\bm{a}=(a_{1},\dots,a_{n})\in\ZB_{\geq 0}^n$, $c_{\bm{a}}\in\CB$,
$x^{\bm{a}}=x_{1}^{a_{1}}\cdots x_{n}^{a_{n}}$, and
$\partial^{\bm{a}}=\partial_{1}^{a_{1}}\cdots \partial_{n}^{a_{n}}$.
Note that $(cf)^{\ast}=\overline{c}f^{\ast}$,
$w\cdot(f^{\ast})=(w\cdot f)^{\ast}$ and
$w\cdot (f^{\ast}g)=(w\cdot f)^{\ast}(w\cdot g)$ for
$c\in\CB$, $w\in W$, $f,g\in S$.

Flatto and Wiener introduced canonical systems to solve
a mean value problem related to vertices for polytopes in
\cite{Fla1,Fla2,Fla-Wei}.
They proved that there exists a canonical system
for every finite real reflection group.
Later in \cite{Iwa}, Iwasaki gave a new definition of the canonical systems
as well as explicit formulas for canonical systems for some types
of reflection groups.
The first and the third authors, in their previous work \cite{NT},
obtained an explicit formula for a canonical system
for every reflection group.
In this article, we extend the definition of canonical systems to
the finite {\it unitary} reflection groups as follows.
\begin{Defi}\label{def-can-sys}
A system $\{f_{1},\dots,f_{n}\}$ of basic invariants
is said to be {\bf canonical} if it satisfies the following
system of partial differential equations:
\begin{align}
f_{i}^{\ast}f_{j}=\delta_{ij}\label{def-cbi}
\end{align}
for $i,j=1,\dots,n$, where $\delta_{ij}$ is the Kronecker delta.
\end{Defi}
Our main result is the following existence theorem.
\begin{Them}\label{Thoerem-cansys}
There exists a canonical system for every finite unitary reflection group.
\end{Them}
Our proof of Theorem \ref{Thoerem-cansys} is classification free.
Furthermore, we give an explicit formula (Theorem \ref{const}) for
a canonical system which is also classification free.
This formula is the same as one obtained in \cite{NT} for the real case,
and we improve the proof.

The organization of this article is as follows.
In Section \ref{sec-basic-inv}, we introduce Lemma \ref{basic-invariants}
which will play an important role in Section \ref{sec-can-exist} when we prove
the existence theorem \ref{Thoerem-cansys}.
In Section \ref{sec-explicit}, we give an explicit formula
for a canonical system.

\section{Basic invariants}\label{sec-basic-inv}
Let $R_{+}$ be the ideal of $R$ generated by homogeneous elements of
positive degrees, and $I=SR_{+}$ the ideal of $S$ generated by $R_{+}$.
We define a unitary inner product
$\langle\cdot,\cdot\rangle:S\times S\rightarrow \CB$ by
\begin{align}
\langle f,g \rangle = f^{\ast}g|_{x=0}=\overline{f}(\partial)g|_{x=0}
\qquad (f,g\in S),\label{inner}
\end{align}
where $x=(x_{1},\dots,x_{n})$ and
$\partial=(\partial_{1},\dots,\partial_{n})$.
Let $e_{H}$ be the order of the cyclic group generated
by the reflection $s_{H}\in W$ corresponding to
a reflecting hyperplane $H$.
Fix $L_{H}\in V^{\ast}$ satisfying $\ker L_{H}=H$.
Let $\Delta$ denote the product of $L_{H}^{e_{H}-1}$ as
$H$ runs over the set of all reflecting hyperplanes.
Then $\Delta$ is skew-invariant, i.e., $w\cdot\Delta=(\det w)\Delta$
for any $w\in W$. Set
\begin{align}
\HC:=\{f^{\ast}\Delta\mid f\in S\}.
\end{align}
The following lemma is obtained by Steinberg \cite{Ste}.
\begin{Lem}\label{Stein1}
Let $f\in S$ be a homogeneous polynomial. Then we have the following:
\begin{enumerate}
\item[$(1)$] $f\in I$ if and only if $f^{\ast}\Delta =0$,
\item[$(2)$] $g^{\ast}f=0$ for all $g\in I$
if and only if $f\in\HC$.
\end{enumerate}
\end{Lem}
It follows from Lemma \ref{Stein1} (2) that $I$ is the orthogonal complement
of $\HC$ with respect to the inner product \eqref{inner} degreewise.

In the rest of this paper, we assume that $W$ acts on $V$ irreducibly.
We fix a $W$-stable graded subspace $U$ of $S$ such that $S=I\oplus U$.
It is known that the $U$ is isomorphic to the regular
representation of $W$
(see Bourbaki \cite[Chap. 5 $\S$5 Theorem 2]{Bou}.)
Hence the multiplicity of $V$ in $U$ is equal to $\dim_{\CB}V=n$.
Let $\pi\,:\,S\rightarrow U$ be
the second projection with respect to the decomposition $S=I\oplus U$.
Then $\pi$ is a $W$-homomorphism.
Let $\{h_{1},\dots,h_{n}\}$ be a system of basic invariants with
$\deg h_{1}\leq\cdots\leq\deg h_{n}$.
The multiset of degrees $m_{i}:=\deg h_{i}\ (i=1,\dots,n)$ does not
depend on a choice of basic invariants.

There exists a unique linear map
$d\,:\,S\rightarrow S\otimes_{\CB}V^{\ast}$ satisfying
$d(fg)=fd(g)+gd(f)$ for $f,g\in S$ and
$dL:=1\otimes L\in\CB\otimes_{\CB}V^{\ast}$
for $L\in V^{\ast}$. The map $d$ is called the differential map.
The differential $1$-form $dh$ is expressed as
$$
dh=\sum_{j=1}^n \partial_{j}h\otimes x_{j}=
\sum_{j=1}^n (\partial_{j}h) dx_{j}
$$
for $h\in S$. Note that $dh$ is invariant if $h$ is invariant.
Define a $W$-homomorphism
$$
\varepsilon\,:\,(S\otimes_{\CB}V^{\ast})^W
\rightarrow R_{+}
$$
by
$$
\varepsilon\left(\sum_{j=1}^n h_{j}dx_{j}\right)=\sum_{j=1}^n x_{j}h_{j}.
$$
Then $\varepsilon\circ d(h)=(\deg h)h$ for any homogeneous polynomial $h$.
The projection $\pi\,:\,S\rightarrow U$ is extended to a $W$-homomorphism
$\tilde{\pi}\,:\,(S\otimes V^{\ast})^W\rightarrow(U\otimes V^{\ast})^W$
defined by
$\tilde{\pi}(\sum_{j=1}^n g_{j}\otimes x_{j})
:=\sum_{j=1}^n\pi(g_{j})\otimes x_{j}$.

\begin{Lem}\label{basic-invariants}
Let $\{h_{1},\dots,h_{n}\}$ be a system of basic invariants.
Put $f_{i}:=(\varepsilon\circ\tilde{\pi})(dh_{i})$,
and $\{f_{1},\dots,f_{n}\}$ is a system of basic invariants.
\end{Lem}
\begin{proof}
Since $h_{1},\dots,h_{n}$ are invariants, so are the $1$-forms
$dh_{1},\dots,dh_{n}$. Thus each $f_{i}$ is invariant because
both $\varepsilon$ and $\tilde{\pi}$ are $W$-homomorphisms.

Next we prove that $\{f_{1},\dots,f_{n}\}$ is a system of basic invariants.
Define $f_{ij}:=\pi(\partial_{j}h_{i})$.
Then $f_{i}=\sum_{j=1}^{n}x_{j}f_{ij}$.
For $j=1,\dots,n$, we express
$\partial_{j}h_{i}=f_{ij}+r_{ij}$ for some $r_{ij}\in I$.
Then $r_{ij}=\sum_{k=1}^{\ell}h_{k}g_{ijk}$
for some $g_{ijk}\in S$. Put $r_{i}:=\sum_{j=1}^n x_{j}r_{ij}$
for $i=1,\dots,n$. Then we have
$$
m_{i}h_{i}=\sum_{j=1}^n x_{j}\partial_{j}h_{i}=f_{i}+r_{i}.
$$
Since $f_{i}$ is invariant, the polynomial $r_{i}=m_{i}h_{i}-f_{i}$
is also invariant. This implies
$$
r_{i}=r_{i}^{\sharp}=
\sum_{k=1}^{\ell}\left(\sum_{j=1}^n x_{j}g_{i,j,k}\right)^{\sharp}h_{k}
\in I^2\cap R,
$$
where ${\sharp}$ denotes the averaging operator, i.e.,
$$
f^{\sharp}=\frac{1}{|W|}\sum_{w\in W}w\cdot f
$$
for $f\in S$. Thus we have $\partial_{j}r_{i}\in I$
for $i,j\in\{1,\dots,n\}$.
Let $J(g_{1},\dots,g_{n})$ denote the Jacobian for $g_{1},\dots,g_{n}\in S$.
Then
\begin{align*}
J(m_{1}h_{1},\dots,m_{n}h_{n})=&\det[\partial_{j}f_{i}+\partial_{j}r_{i}]_{i,j}\\
\equiv&\det[\partial_{j}f_{i}]_{i,j}=J(f_{1},\dots,f_{n})\pmod{I}.
\end{align*}
It immediately follows that $\Delta\notin I$ by Lemma \ref{Stein1} (1).
Since $J(h_{1},\dots,h_{n})$ is a nonzero constant multiple of $\Delta$,
we obtain $J(f_{1},\dots,f_{n})\notin I$, and thus
$J(f_{1},\dots,f_{n})\neq 0$.
By the Jacobian criterion (e.g., \cite[Proposition 3.10]{Hum}),
$\{f_{1},\dots,f_{n}\}$ is algebraically independent.
Therefore $\{f_{1},\dots,f_{n}\}$ is a system of basic invariants
because $\deg f_{i}=\deg h_{i}$ for $i=1,\dots,n$.
\end{proof}

\begin{flushleft}
{\it Remark.}
There exists a $W$-stable subspace $U^{\prime}$ of $S$ such that
$S=I\oplus U^{\prime}$ and
$df_{1},\dots,df_{n}\in(U^{\prime}\otimes_{\CB}V^{\ast})^W$.
However, we do not know whether $U^{\prime}$ coincides with $U$ or not.
In section \ref{sec-can-exist}, we see that $U$ and $U^{\prime}$ coincide
when $U=\HC$.
\end{flushleft}

\section{Existence of a canonical system}\label{sec-can-exist}
In this section, we prove
Theorem \ref{Thoerem-cansys}
which is the existence theorem
of a canonical system.
The following lemma is widely known.
\begin{Lem}\label{weyl}
Let $g\in S$ be a homogeneous polynomial, and put
$g_{j}:=\partial_{j}g$ for $j=1,\dots,n$. Then, for any $h\in S$, we have
\begin{align}
g^{\ast}(x_{j}h)=x_{j}g^{\ast}h+g_{j}^{\ast}h.\label{deff-ope}
\end{align}
\end{Lem}
\begin{proof}
We only need to verify the assertion when $g$ is a monomial.
We verify it by induction on $\deg g$.
Let $\bm{a}=(a_{1},\dots,a_{n})$ be a multi-index with $|\bm{a}|=\deg g$.
Then
\begin{align*}
\partial^{\bm{a}}(x_{j}h)&=\partial^{\bm{a}-\bm{e}_{j}}\partial_{j}(x_{j}h)
=\partial^{\bm{a}-\bm{e}_{j}}h+
\partial^{\bm{a}-\bm{e}_{j}}(x_{j}\partial_{j}h)\\
&=\partial^{\bm{a}-\bm{e}_{j}}h+
\left(x_{j}\partial^{\bm{a}-\bm{e}_{j}}\partial_{j}h+
(a_{j}-1)\partial^{\bm{a}-2\bm{e}_{j}}\partial_{j}h\right)\\
&=x_{j}\partial^{\bm{a}}h+a_{j}\partial^{\bm{a}-\bm{e}_{j}}h.
\end{align*}
\end{proof}

By Lemma \ref{Stein1}, $I$ is the orthogonal complement of
$\HC$ with respect to the inner product \eqref{inner}, and
the $W$-stable graded space $S$ is decomposed into
the direct sum of the $W$-stable graded subspaces
$I$ and $\HC$, i.e., $S=I\oplus\HC$.
Let $\pi\,:\,S\rightarrow \HC$ be the second projection with respect to
the decomposition $S=I\oplus\HC$.
Let $h_{1},\dots,h_{n}$ be an arbitrary system of basic invariants.
Put $f_{ij}:=\pi(\partial_{j}h_{i})\in\HC=\{f^{\ast}\Delta\mid f\in S\}$ for $i,j=1,\dots,n$,
and $f_{i}:=(\varepsilon\circ\tilde{\pi})(dh_{i})=\sum_{j=1}^n x_{j}f_{ij}$
for $i=1,\dots,n$. Then $\{f_{1},\dots,f_{n}\}$ is a system
of basic invariants by Lemma \ref{basic-invariants}.
We are now ready to give a proof of Theorem \ref{Thoerem-cansys}.

Let $g\in R_{+}$ be a homogeneous invariant polynomial with $\deg g<m_{i}$.
By using Lemma \ref{Stein1} and Lemma \ref{weyl} we obtain
\begin{align*}
g^{\ast}f_{i}&=\sum_{j=1}^n g^{\ast}(x_{j}f_{ij})
=\sum_{j=1}^n\left(x_{j}g^{\ast}f_{ij}+g_{j}^{\ast}f_{ij}\right)
=\sum_{j=1}^n g_{j}^{\ast}f_{ij}\in\HC.
\end{align*}
Meanwhile, $g^{\ast}f_{i}$ is an invariant polynomial of positive degree
since $g$ and $f_{i}$ are invariant and $\deg g<m_{i}$.
Therefore we have
$$
g^{\ast}f_{i}\in\HC\cap I=\{0\}.
$$
In particular, when $g=f_{j}$ with $\deg f_{j}<m_{i}$,
we have $f_{j}^{\ast}f_{i}=0$.
It immediately follows that
$f_{j}^{\ast}f_{i}=0$ if $\deg f_{j}>\deg f_{i}$.
Applying the Gram-Schmidt orthogonalization with respect to the inner product
\eqref{inner}, we obtain a canonical system of basic invariants.
This completes our proof of Theorem \ref{Thoerem-cansys}.
\smallskip

The subspace spanned by a canonical system
can be characterized
as follows:

\begin{Prop}\label{chara-can-sys}
Let
$\{f_{1},\dots,f_{n}\}$ be a canonical system
and
$\FC:=\langle f_{1},\dots,f_{n}\rangle_{\CB}$.
Then
\begin{enumerate}
\item[$(1)$] $\FC=
\bigoplus_{k=1}^{\infty}
\{f\in R_{k}\mid g^{\ast}f=0\ {\rm for}\ g\in R_{\ell}
\ {\rm with}\ 0<\ell<k\}
$,
\label{Fteisiki}
where $R_{k}:=R\cap S_{k}$,
\item[$(2)$] $
\langle df_{1},\dots,df_{n}\rangle_{\CB}=
(\HC\,\otimes_{\CB}\,V^{\ast})^W,$
\label{HtimesV=df}
\item[$(3)$] $\FC=\varepsilon((\HC\otimes_{\CB}V^{\ast})^W)$
\label{F=can-sys}.
\end{enumerate}

\end{Prop}
\begin{proof}
Define
\[
\GC:=
\bigoplus_{k=1}^{\infty}
\{f\in R_{k}\mid g^{\ast}f=0\ {\rm for}\ g\in R_{\ell}
\ {\rm with}\ 0<\ell<k\}.
\]
Let $f\in\GC$ be a homogeneous polynomial.
For any $g\in I$, it is not hard to see that
$g^{\ast}(\partial_{j}f)=\partial_{j}(g^{\ast}f)=0$.
Thus we have $\partial_{j}f\in\HC$ for $j=1,\dots,n$ by Lemma \ref{Stein1}.
This implies $d(\GC)\subseteq(\HC\otimes_{\CB}V^{\ast})^W$.
The inclusion $\FC\subseteq\GC$
follows immediately because $\{f_{1},\dots,f_{n}\}$ is a canonical system.
Hence we have the following commutative diagram:
\begin{align}\label{seq-f_i-G-HxV}
\begin{array}{ccc}
\GC & \overset{d}{\rightarrow} & (\HC\otimes_{\CB}V^{\ast})^W \\
\rotatebox{90}{$\subseteq$} & & \rotatebox{90}{$\subseteq$} \\
\FC & \overset{d}{\rightarrow} &
\langle df_{1},\dots,df_{n}\rangle_{\CB}.
\end{array}
\end{align}

Since $\ker(d)=\CB$ and $\GC$ does not contain any nonzero constant,
the horizontal maps $d$ are both injective.
Recall that $V$ is an irreducible representation
and that
$\HC$ affords the regular representation of $W$.
Hence we have
$\dim (\HC\,\otimes_{\CB}\,V^{\ast})^W=\dim V = n$
because
$(\HC\,\otimes_{\CB}\,V^{\ast})^W\simeq\Hom_{W}(V,\HC)$;
this isomorphism is referred from \cite{Or-Solo}
or the proof of \cite[Lemma 6.45]{Orlik-Terao}.
By comparing the dimensions, we have $\FC=\GC$, which is (1).
Sending both sides of this equality by $d$, we obtain
\begin{align}
\langle df_{1},\dots,df_{n}\rangle_{\CB}=
d(\FC)=d(\GC)
=(\HC\otimes_{\CB}V^{\ast})^W.\label{seq=df_i=dG=HxV}
\end{align}
So the equality $(2)$ is proved.
Moreover, we have
\begin{align}
\langle f_{1},\dots,f_{n}\rangle_{\CB}
=
\FC=\GC
=\varepsilon((\HC\otimes_{\CB}V^{\ast})^W)
\label{seq-G=f_i=eHxV=F}
\end{align}
by applying $\varepsilon$ to \eqref{seq=df_i=dG=HxV}.
This verifies $(3)$.
\end{proof}

\section{An explicit construction of a canonical system}\label{sec-explicit}
The following is a key to our explicit formula
for
a canonical system.
\begin{Defi}[c.f. \cite{NT}]\label{def-keymap}
Define a linear map
\begin{align*}
\phi\,:\, S\rightarrow\HC
\end{align*}
by
$\phi(f):=(f^{\ast}\Delta)^{\ast}\Delta$
for $f \in S$.
The map $\phi$ induces a $W$-homomorphism
$$
\tilde{\phi}\,:\,(S\otimes V^{\ast})^W\rightarrow(\HC\otimes V^{\ast})^W
$$
defined by $\tilde{\phi}(\sum f\otimes x):=\sum\phi(f)\otimes x$.
\end{Defi}
One has
\begin{align*}
w\cdot\phi(f)&=((w\cdot f)^{\ast}(w\cdot\Delta))^{\ast}(w\cdot\Delta)
=((w\cdot f)^{\ast}(\det(w)\Delta))^{\ast}(\det(w)\Delta)\\
=&\overline{\det(w)}\det(w)((w\cdot f)^{\ast}\Delta)^{\ast}(\Delta)
=\phi(w\cdot f)
\end{align*}
for $w\in W$ and $f\in S$.
Therefore $\phi$ is a $W$-homomorphism, and so is $\tilde{\phi}$.

Let $\{h_{1},\dots,h_{n}\}$ be an arbitrary system of basic invariants,
and assume $\deg h_{i}=m_{i}$ for $i=1,\dots,n$.
Let $\{f_{1},\dots,f_{n}\}$ be a canonical system with
$\deg f_{i}=m_{i}$.
We have already shown
in Proposition \ref{chara-can-sys} (1)
that
$(\HC\otimes V^{\ast})^W=
\langle df_{1},\dots,df_{n}\rangle_{\CB}.$
\begin{Lem}\label{phi(N)neq0}.
The restriction
$$
\tilde{\phi}|_{\langle dh_{1},\dots,dh_{n}\rangle_{\CB}}:
\langle dh_{1},\dots,dh_{n}\rangle_{\CB}\rightarrow
(\HC\otimes V^{\ast})^W=\langle df_{1},\dots,df_{n}\rangle_{\CB}
$$
is isomorphic.
\end{Lem}
\begin{proof}
It is enough to prove the injectivity.
Fix $h\in\langle h_{1},\dots,h_{n}\rangle_{\CB}$
with
$\tilde{\phi}(dh)=0$. It follows from Lemma \ref{Stein1} that
$\ker\tilde{\phi}=(I\otimes_{\CB}V^{\ast})^W$.
Then we have $dh\in(I\otimes_{\CB}V^{\ast})^W$.
At the same time, since $\{f_{1},\dots,f_{n}\}$ is a system of
basic invariants, we may write
\begin{align*}
h=\sum_{k=1}^{n}\lambda_{k}f_{k}+P,
\end{align*}
where $\lambda_{k}\in\CB$ and $P\in I^2\cap R$.
Then the $1$-form $dP$ lies in $(I\otimes_{\CB}V^{\ast})^W$,
and $df_{k}\in(\HC\otimes_{\CB}V^{\ast})^W$ for $k=1,
\dots,n$
by Proposition \ref{chara-can-sys}. Hence we have
$$
\sum_{k=1}^n \lambda_{k}df_{k}=dh-dP
\in(\HC\otimes_{\CB}V^{\ast})^W\cap(I\otimes_{\CB}V^{\ast})^W=\{0\}.
$$
This implies $\lambda_{k}=0$ for all $k=1,\dots,n$
since $\{df_{1},\dots,df_{n}\}$ is linearly independent over $\CB$.
Thus we have $h=P\in I^2\cap R$.
The algebraic independence of $h_{1},\dots,h_{n}$
implies $\langle h_{1},\dots,h_{n}\rangle\cap I^2=\{0\}$.
Therefore $h=0$.
\end{proof}
The image of
$\tilde{\phi}|_{\langle dh_{1},\dots,dh_{n}\rangle_{\CB}}$
coincides with $\langle df_{1},\dots,df_{n}\rangle_{\CB}$
by Lemma \ref{phi(N)neq0}.
Therefore we have a chain of the linear maps:
\begin{align}
\langle h_{1},\dots,h_{n}\rangle_{\CB}
\overset{d}{\rightarrow}
\langle dh_{1},\dots,dh_{n}\rangle_{\CB}
\overset{\tilde{\phi}}{\rightarrow}
\langle df_{1},\dots,df_{n}\rangle_{\CB}
\overset{\varepsilon}{\rightarrow}
\langle f_{1},\dots,f_{n}\rangle_{\CB}.\label{epd}
\end{align}
The image of $\{
h_{1},\dots,h_{n}\}$ by
the composition of all the maps in \eqref{epd} forms a basis for
$\langle f_{1},\dots,f_{n}\rangle_{\CB}$. Thus we have the following
explicit formula for a canonical system of basic invariants. 
\begin{Them}[c.f. \cite{NT}]\label{const}
Let $h_{1},\dots,h_{n}$ be an arbitrary system of basic invariants.
Applying the Gram-Schmidt orthogonalization to
\begin{align*}
\left\{\varepsilon\circ\tilde{\phi}(dh_{i})
=\sum_{j=1}^n x_{j}\phi(\partial_{j}h_{i})\,\middle|\,
i=1,\dots,n\right\}
\end{align*}
with respect to the inner product \eqref{inner}, we obtain
a canonical system of basic invariants.
\end{Them}
\begin{flushleft}
{\bf Remark.}
Theorem \ref{const} asserts the same formula as
\cite[Theorem 3.4]{NT} for the real case.
In \cite{NT}, to prove the theorem, we showed the symmetricity of $\phi$
with respect to the inner product \eqref{inner} and
considered eigenvectors of $\tilde{\phi}$.
In contrast, the proof of this paper does not require these arguments.
\end{flushleft}

\end{document}